\documentclass{article}
\usepackage{amsmath}
\usepackage{amssymb}
\usepackage{latexsym}
\usepackage{amsthm}
\usepackage{mathrsfs}
\usepackage{eucal}
\usepackage{wasysym}
\usepackage{fancyhdr}
\usepackage{amsfonts}
\usepackage{amssymb}
\usepackage{epsfig}
\usepackage{epstopdf}

\usepackage{tkz-graph}
\usetikzlibrary{arrows}
\usepackage{tkz-euclide}
\usetkzobj{all}

\newtheorem{theorem}{Theorem}[section]

\newtheorem{lemma}[theorem]{Lemma}
\newtheorem{proposition}[theorem]{Proposition}
\newtheorem{conj}[theorem]{Conjecture}

\newcommand{\ZZ}{\mathbf{ZZ}}
\newcommand{\QQ}{\mathbf{C}}

\begin{document}

\title{Resonance Graphs and Perfect Matchings of Graphs on Surfaces}

 \author{ Niko Tratnik\thanks{Faculty of Natural Sciences and Mathematics, University of Maribor, Slovenia. Email: niko.tratnik@um.si. Supported by  the Slovenian Research Agency.} \ and
 Dong Ye\thanks{Department of Mathematical Sciences and Center for Computational Sciences, Middle Tennessee State University,
Murfreesboro, TN 37132, USA. Email: dong.ye@mtsu.edu. Partially supported by a grant from the Simons Foundation (No. 359516).}}


\maketitle

\begin{abstract}

Let $G$ be a graph embedded in a surface and let $\mathcal F$ be a set of even faces of $G$ (faces bounded by a cycle of even length). The resonance graph of $G$ with respect to $\mathcal F$, denoted by  $R(G;\mathcal F)$, is a graph such that its vertex set is the set of all perfect matchings of $G$ and two vertices $M_1$ and $M_2$ are adjacent to each other if and only if the symmetric difference $M_1\oplus M_2$ is a cycle bounding some face in $\mathcal F$. It has been shown that if $G$ is a matching-covered plane bipartite graph, the resonance graph of $G$ with respect to the set of all inner faces is isomorphic to the covering graph of a distributive lattice. It is evident that the resonance graph of a plane graph $G$ with respect to an even-face set
$\mathcal F$ may not be the covering graph of a distributive lattice. In this paper, we show the resonance graph of a graph $G$ on a surface with respect to a given even-face set $\mathcal F$ can always be embedded into a hypercube as an induced subgraph. Furthermore, we show that the Clar covering polynomial of $G$ with respect to $\mathcal F$ is equal to the cube polynomial of the resonance graph  $R(G;\mathcal F)$, which generalizes previous results on some subfamilies of plane graphs.

\medskip

\end{abstract}

\section{Introduction}

Unless stated otherwise, the graphs considered in this paper are simple and finite. Let $G$ be a graph with vertex set $V(G)$ and edge set $E(G)$. A {\em perfect matching} of $G$ is a set of independent edges of $G$ which covers all vertices of $G$. In other words, the edge induced subgraph of a perfect matching is a spanning 1-regular subgraph, also called {\em 1-factor}. Denote the set of all perfect matchings of a graph $G$ by $\mathcal M(G)$.

A \textit{surface} in this paper always means a closed surface which is a compact and connected 2-dimensional manifold without boundary. An embedding of a graph $G$ in a surface $\Sigma$ is an injective mapping which maps $G$ into the surface  $\Sigma$ such that a vertex of $G$ is mapped to a point and an edge is mapped to a simple path connecting two points corresponding to two end-vertices of the edge. Let $G$ be a graph embedded in a surface $\Sigma$. For convenience, a {\em face} of $G$ is defined as the closure of a connected component of $\Sigma \backslash G$. The boundary of a face $f$ is denoted by $\partial f$ and the set of edges on the boundary of $f$ is denoted by $E(f)$. 
An embedding of $G$ is a {\em 2-cell} or {\em cellular} embedding if every face is homomorphic to a close disc. Note that, every connected graph admits a 2-cell embedding in a closed surface. A cellular embedding of $G$ on $\Sigma$ is a {\em strong embedding} (or closed 2-cell embedding) if the boundary of every face is a cycle.  
Denote the set of all faces of a graph $G$ on a surface by $\CMcal F(G)$. A face $f$ of $G$ is {\em even} if it is bounded by an even cycle and a set of even faces is also called an \textit{even-face set}. A cycle is a {\em facial cycle} of $G$ if it is the boundary of a face. If there is no confusion, a graph $G$ on a surface always means an embedding of $G$ in the surface, and a face sometime means its boundary cycle.

Let $G$ be a graph embedded in a surface $\Sigma$ with a perfect matching $M$. A cycle $C$ of $G$ is {\em $M$-alternating} if the edges of
$C$ appear alternately between $M$ and $E(G)\backslash M$.  Moreover, a face $f$ of $G$ is \textit{$M$-alternating} if  $\partial f$ is an $M$-alternating cycle.
For a given set of even faces $\mathcal F\subseteq \CMcal F(G)$, the {\em resonance graph of $G$ with respect to $\mathcal F$} (or Z-transformation graph of $G$ \cite{zh-gu-ch}), denoted by $R(G; \mathcal F)$, is a graph with vertex set $\mathcal M(G)$ such that two vertices $M_1$ and $M_2$ are adjacent if and only if the symmetric difference $M_1 \oplus M_2=(M_1\cup M_2)\backslash (M_1\cap M_2)$ is the boundary of a face $f\in \mathcal F$, i.e. $E(f)=M_1\oplus M_2$.

Resonance graph was first introduced for hexagonal systems (also called benzenoid systems) which are plane bipartite graphs with only hexagons as inner faces in \cite{grun-82}, and reintroduced by Zhang, Guo and Chen \cite{zh-gu-ch} in the name Z-transformation graph, and has been extensively studied for hexagonal systems \cite{kl-zi-bri, khaled, skvz, zhang-lam-shiu}.  Later, the concept was extended to all plane bipartite graphs by Lam and Zhang \cite{lam-zhang, zhang-10} and fullerenes \cite{DTYZ,tr-zp-2}.

A graph $G$ is {\em elementary}  if the edges of $G$ contained in a perfect matching induce a connected subgraph (cf. \cite{LP}). An elementary bipartite graph is also matching-covered (or 1-extendable), i.e., every edge is contained in a perfect matching. It is known that a matching-covered graph is always 2-connected \cite{P}. So is an elementary bipartite graph. It has been shown in \cite{zhang} that a plane bipartite graph $G$ is elementary if and only if each face boundary of $G$ is an $M$-alternating cycle for some perfect matching $M$ of $G$ . 

\begin{theorem}[Lam and Zhang, \cite{lam-zhang}]\label{thm:plane}
Let $G$ be a plane elementary bipartite graph and $\mathcal F$ be the set of all inner faces of $G$. Then $R(G;\mathcal F)$ is the covering graph of a distributive lattice. 
\end{theorem}

The above result was also obtained independently by Propp \cite{propp}.  Similar operations on other combinatorial structures such as, spanning trees, orientations and flows have been discovered to have similar properties \cite{ FK,propp}. Moreover, such distributive lattice structure is also established on the set of all perfect matchings of open-ended carbon nanotubes \cite{tr-zi}. However, if $G$ is not plane bipartite graph, the resonance graph of $G$ with respect to some set of even faces may not be the covering graph of a distributive lattice (cf. \cite{tr-zp-2}). It was conjectured in \cite{tr-zp-2} that every connected component of the resonance graph of a fullerene is a median graph and these graphs are a family of well-studied graphs (see \cite{br-kl-skr-2003, br-kl-skr-2006, kl-mu-skr}), containing the covering graphs of distributive lattices. Median graphs are a subfamily of cubical graphs \cite{agnar, ga-gr, ov}, which are defined as subgraphs of hypercubes, and have important applications   in coding theory, data transmission, and linguistics (cf. \cite{ga-gr, ov}). 

In this paper, we consider the resonance graphs of graphs embedded in a surface in a very general manner. 
The following is one of the main results.

\begin{theorem}\label{main1}
Let $G$ be a graph embedded in a surface and let $\mathcal F\ne \CMcal F(G)$ be an even-face set. Then every connected component of the resonance graph $R(G;{\mathcal F})$ is an induced cubical graph.
\end{theorem}

A {\em 2-matching} of a graph $G$ embedded in a surface is a spanning subgraph consisting of independent edges and cycles. A {\em Clar cover} of $G$ is a 2-matching in which every cycle is a facial cycle. It has been evident that the enumeration of Clar covers of a molecular graph has physical meaning in chemistry and statistical physics. The {\em Clar covering polynomial} or {\em Zhang-Zhang polynomial} of graph $G$ embedded in a surface is a polynomial used to enumerate all Clar covers of $G$. Zhang-Zhang polynomial was introduced in \cite{zh-zh} for hexagonal systems. A  definition of Zhang-Zhang polynomial will be given in the next section.  Zhang et al.\,\cite{zhang-13} demonstrate the equivalence between the Clar covering polynomial of a hexagonal system and the cube polynomial of its resonance graph, which is further generalized to spherical hexagonal systems by Berli\v{c} et al.\,\cite{be-tr-zi} and to fullerenes \cite{tr-zp-2}. In this paper, we show the equivalence between the Zhang-Zhang polynomial of a graph $G$ embedded in a surface and the cube polynomial of its resonance graph as follows.

\begin{theorem}  \label{main2}
Let $G$ be a graph embedded in a surface and let $\mathcal F \ne \CMcal F(G)$ be a set of even faces. Then the  Zhang-Zhang polynomial of $G$ with respect to $\mathcal F$ is equal to the cube polynomial  of the resonance graph $ R(G;\mathcal F)$.
\end{theorem}

The paper is organized as follows: some detailed definitions are given in Section 2, the proofs of Theorem~\ref{main1} and    Theorem~\ref{main2} are given in Section~3 and Section~4, respectively. We conclude the paper with some problems as Section~5. 

\section{Preliminaries}

Let $G$ be a graph and let $u,v$ be two vertices of $G$. The {\em distance} between $u$ and $v$, denoted by $d_G(u,v)$  (or $d(u,v)$ if there is no confusion)
is the length of a shortest path joining $u$ and $v$.
A {\em median} of a triple of vertices $\{u, v, w\}$ of $G$ is a vertex $x$ that lies on a
shortest $(u,v)$-path, on a shortest $(u, w)$-path and on a shortest $(v, w)$-path. Note that $x$ could be
one vertex from $\{u, v, w\}$. A graph is a {\em median graph} if every triple of vertices has a
unique median. Median graphs were first introduced by Avann \cite{av}. Median graphs arise
naturally in the study of ordered sets and distributive lattices. A {\em lattice} is a poset such that any two 
elements have a greatest lower bound and a least upper bound.  The {\em covering graph} of a  lattice $\CMcal L$ is a graph whose vertex set
consists of all elements in $\CMcal L$ and two vertices $x$ and $y$ are adjacent if and only if either $x$ covers $y$ or $y$ covers $x$. 
A {\em distributive lattice} is a lattice in which the operations of the join and meet distribute  over each other.
It is known that the covering graph of a distributive lattice is a median graph but not vice versa \cite{DR}. 


The $n$-{\em dimensional hypercube} $Q_n$ with $n \geq 1$, is the graph whose vertices are all binary strings of length $n$ and two vertices are adjacent if and only if their strings differ exactly in one position. For convenience, define $Q_0$ to be the one-vertex graph.
The {\em cube polynomial} of a graph $G$  is defined  as follows,
$$\QQ(G,x)=\sum_{i\geq 0} \alpha_i (G)x^i,$$
where $\alpha_i(G)$ denotes the number of induced subgraphs of $G$ that are isomorphic to the $i$-dimensional hypercube. The cube 
polynomials of median graphs have been studied by Bre\v{s}ar, Klav\v{z}ar and \v{S}krekovski \cite{br-kl-skr-2003,br-kl-skr-2006}.

Let $H$ and $G$ be two graphs. A function $\ell: V(H) \rightarrow V(G)$ is called an \textit{embedding of $H$ into $G$} if $\ell$ is injective and, for any two vertices $x,y \in V(H)$, $\ell(x)\ell(y) \in E(G)$ if $xy \in E(H)$. If such a function $\ell$ exists, we say that $H$ can be \textit{embedded} in $G$. In other words, $H$ is a subgraph of $G$. Moreover, if $\ell$ is an embedding such that for any two vertices $x,y \in V(H)$, $ \ell(x)\ell(y) \in E(G)$ if and only if $xy \in E(H)$, then \textit{$H$ can be embedded in $G$ as an induced subgraph.} An embedding $\ell$ of graph $H$ into graph $G$ is called an \textit{isometric embedding} if for any two vertices $x,y \in V(H)$ it holds $d_H(x,y) = d_G(\ell(x),\ell(y))$.
A graph $H$ is ({\em induced}) \textit{cubical} if $H$ can be embedded into $Q_n$ for some integer $n\ge 1$ (as an induced subgraph), and $H$ is called a \textit{partial cube} if $H$ can be isometrically embedded into $Q_n$ for some integer $n\ge 1$. For more information and properties on cubical graphs, readers may refer to \cite{agnar,br-im-kl-mu-sk, br-kl-li-mo, ga-gr}. It holds that a median graph is a partial cube (in fact, even stronger result is true, i.e. median graphs are retracts of hypercubes, see \cite{bandelt}). Therefore,
we have the nested relations for these interesting families of graphs: 
\[\{\mbox{covering graphs of distributive lattices}\}\subsetneq \{\mbox{median graphs}\}\subsetneq\{\mbox{partial cubes}\} \subsetneq \]
\[ \subsetneq \{\mbox{induced cubical graphs}\} \subsetneq \{\mbox{cubical graphs}\}.\]

In the following, let $G$ be a graph embedded in a surface $\Sigma$ and let $f$ be a face bounded by a cycle of $G$. If $G$ has two perfect matchings
$M_1$ and $M_2$ such that the symmetric difference $M_1\oplus M_2$ is a cycle which bounds face $f$, then we say that $M_1$ can be obtained from $M_2$ by \textit{rotating} the edges of $f$. Therefore, two perfect matchings $M_1$ and $M_2$  of $G$  are adjacent in the resonance graph $R(G;\mathcal F)$ if and only if $M_1$ can be obtained from $M_2$ by rotating the edges of some face $f\in \mathcal F$. We sometimes also say that \textit{edge $M_1M_2$ corresponds to face $f$} or {\it face $f$ corresponds to edge $M_1M_2$}.

A {\em  Clar cover} of $G$ is a spanning subgraph $S$ of $G$ such that every component of $S$ is either the boundary of an even face or an edge. 
Let $\mathcal F\subseteq \CMcal F(G)$ be an even-face set. The {\em Zhang-Zhang polynomial of $G$ with respect to $\mathcal F$} (also called the Clar covering polynomial, see \cite{zh-zh}) is defined as follows,
\[\ZZ_{\mathcal  F }(G,x)=\sum_{k \geq 0}z_k (G,\mathcal F)x^k,\]
where $z_k(G,\mathcal F)$ is the number of Clar covers of $G$ with exact $k$ faces and all the $k$ faces belong to $\mathcal F$.
 Note that $z_0(G,\mathcal F)$ equals the number of of perfect matchings of $G$, i.e., the number of vertices of the resonance graph $R(G;\mathcal F)$.

\section{Resonance graphs and cubical graphs}

Let $G$ be a graph embedded in a surface and let $\mathcal F$ be a set of even faces of $G$ such that $\mathcal F\ne  \CMcal F(G)$. In this section, we investigate the resonance graph $R(G;\mathcal F)$ and show that every connected component of $R(G;\mathcal F)$ is an induced cubical graph.

\begin{lemma}\label{osnova}
Let $G$ be a graph embedded in a surface and let $\mathcal F\ne \CMcal F(G)$ be an even-face set. Assume that $C = M_0M_1\ldots M_{t-1}M_0$ is a cycle of the resonance graph $R(G;\mathcal F)$. Let $f_i$ be the face of $G$ corresponding to the edge $M_iM_{i+1}$ for $i\in \{0,1,..,t-1\}$ where subscripts take modulo $t$. Then every face of $G$ appears an even number of times in the face sequence $(f_0, f_1,...,f_{t-1})$. 
\end{lemma}

\begin{proof} Let $f$ be a face of $G$, and let $\delta(f)$ be the number of times $f$ appears in the face sequence $(f_0, f_1, \ldots, f_{t-1})$. It suffices
to show that $\delta(f)\equiv 0\pmod 2$.
Since  $C = M_0M_1\ldots M_{t-1}M_0$ is a cycle of $R(G;\mathcal F)$ and $f_i$ is the corresponding face of the edge $M_iM_{i+1}$, it follows that $M_i\oplus M_{i+1}= E(f_i)$ for $i\in \{0,1,..., t-1\}$.   So
\begin{equation}
E(f_0) \oplus E(f_1) \oplus \ldots \oplus E(f_{t-1})=\oplus_{i=0}^{t-1}(M_i\oplus M_{i+1})=\emptyset
\end{equation}
where all subscripts take modulo $t$.

Let $f$ and $g$ be two faces of $G$ such that $E(f)\cap E(g)\ne \emptyset$, and let $e\in E(f)\cap E(g)$. Since $e$ is contained by only $f$ and $g$, and the total number of faces in the sequence $(f_0, f_1,\ldots, f_{t-1})$ containing $e$ is even by (1), it follows that $\delta(f)+\delta(g)\equiv 0\pmod 2$. So $\delta(f)\equiv \delta(g) \pmod 2$. Therefore, all faces $f$ of $G$ have the same parity for $\delta(f)$.

Note that $\mathcal F\ne \CMcal F(G)$. So $G$ has a face $g\notin \mathcal F$. Hence $g$ does not appear in the face sequence. It follows that $\delta(g)=0$. Hence $\delta(f)\equiv \delta(g) \equiv 0\pmod 2$ for any face $f$ of $G$. This completes the proof.
\end{proof}

 The following proposition follows immediately from Lemma~\ref{osnova}.

\begin{proposition}\label{prop:bipartite}
Let $G$ be a graph embedded in a surface, and let $\mathcal F\ne \CMcal F(G)$ be a set of even faces. Then the resonance graph $R(G; \mathcal F)$ is bipartite.
\end{proposition}

\begin{proof}
Let $C=M_0M_1\ldots M_{t-1}M_0$ be a cycle of $R(G;\mathcal F)$ and let $f_i$ be the face corresponding to the edge $M_iM_{i+1}$
for $i\in \{0,...,t-1\}$ (subscripts take modulo $t$). By Lemma~\ref{osnova}, every face $f$ of $G$ appears an even number of times in the sequence $(f_0,f_1,\ldots, f_{t-1})$. So $C$ is a cycle of even length. Therefore, $R(G;\mathcal F)$ is a bipartite graph.
\end{proof}

\begin{figure}[!htb]
\begin{center}
\begin{tikzpicture}[thick, scale=1]
\draw [black, thick] (-3.5,0)--(-2,0)--(-0.5,0);

\draw [black, thick] (-0.5,-1.5)--(-2,-1.5)--(-3.5,-1.5);

\draw [black, double] (-3.5,0)--(-3.5,-1.5);
\draw [black, double] (-2,0)--(-2,-1.5);
\draw [black, double] (-0.5,0)--(-0.5,-1.5);

\draw [black, thick] (4.05,0)--(5,-1.5)--(3.2,-1.5)--(4.05,0);

\filldraw [black] (-3.5,0) circle (3pt);
\filldraw [black] (-3.5,-1.5) circle (3pt);
\filldraw [black] (-2,0) circle (3pt);
\filldraw [black] (-2,-1.5) circle (3pt);
\filldraw [black] (-0.5,0) circle (3pt);
\filldraw [black] (-0.5,-1.5) circle (3pt);

\filldraw [black] (4.05,0) circle (3pt);
\filldraw [black] (3.2,-1.5) circle (3pt);

\filldraw [black] (5,-1.5) circle (3pt);

\node [] at (-2,0.5) { $f_3$};
\node [] at (-2.8,-0.75) { $f_1$};
\node [] at (-1.3,-0.75) { $f_2$};
\node [] at (-2,-2.75) {$G$};

\node [] at (4.05,0.35) { $M_1$};
\node [] at (2.8,-1.75) { $M_3$};

\node [] at (5.5,-1.75) { $M_2$};

\node [] at (5,-0.75) { $f_1$};
\node [] at (4.1,-2) { $f_3$};
\node [] at (3.1,-0.75) { $f_2$};
\node [] at (4.2, -2.75) {$R(G; \CMcal F(G))$};

\end{tikzpicture}
\caption{{\small A non-bipartite resonance graph where the double edges of $G$ form $M_1$.}}\label{non_bip}
\end{center}
\end{figure}
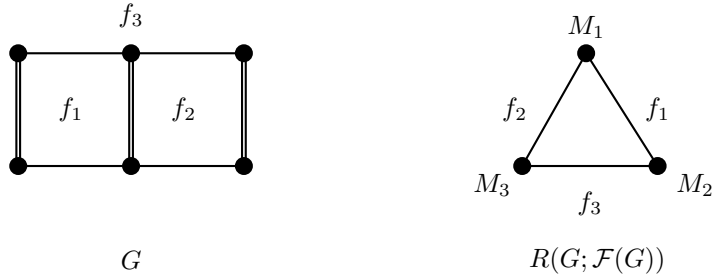

The above proposition shows that the resonance graph $R(G;\mathcal F)$ is bipartite if $\mathcal F\ne \CMcal F(G)$. However, if $\mathcal F= \CMcal F(G)$, then $R(G;\mathcal F)$ may not be bipartite. For example,  the graph $G$ on the left in Figure \ref{non_bip} is a plane graph with three faces $f_1$, $f_2$ and $f_3$. If $\mathcal F=\{f_1, f_2, f_3\}$,
then its resonance graph $R(G;\mathcal F)$ is a triangle as shown on the right in Figure~\ref{non_bip}.

It is known that a resonance graph $R(G;\mathcal F)$ may not be connected \cite{tr-zp-2}. In the following, we
focus on a connected component $H$ of  $R(G;\mathcal F)$, and always assume $\lbrace f_1, \ldots, f_k \rbrace$ to be the set of all the faces that correspond to the edges of $H$, which is a subset of $\mathcal F$. Denote the set of all the edges of $H$ that correspond to the face $f_i$ by $E_{i}$ for $i \in \lbrace 1, \ldots, k \rbrace$. In the rest of this section,  $H \backslash E_{i}$ denotes the graph obtained from $H$ by deleting all the edges from $E_i$. 

\begin{lemma}\label{lem:disconn}
Let $R(G;\mathcal F)$ be the resonance graph of a graph $G$ on a surface with
respect to a set of even faces $\mathcal F\ne \CMcal F(G)$, and let $H$ be a connected component of $R(G;\mathcal F)$.
Assume that $M_1M_2\in E_{i}$ where $E_i$ is the set of all the edges of $H$ corresponding to some face $f_i\in \mathcal F$. Then $M_1$ and $M_2$ belong to different components of
$H \backslash E_{i}$.
\end{lemma}
\begin{proof}
Suppose to the contrary that $M_1$ and $M_2$ belong to the same component of $H \backslash E_i$. Then $H \backslash E_i$ has a path $P$ joining $M_1$ and $M_2$.  In other words, $H$ has a path $P$ joining $M_1$ and $M_2$ such that $E(P)\cap E_{i}=\emptyset$. Let $C=P\cup \lbrace M_1M_2 \rbrace$ be a cycle of $H$. Then $f$ appears exactly once in the face sequence corresponding to the edges in the cycle $C$, which contradicts
Lemma~\ref{osnova}. The contradiction implies that $M_1$ and $M_2$ belong to different components of $H \backslash E_i$.
\end{proof}

By Lemma~\ref{lem:disconn}, the graph $H \backslash E_i$ is disconnected for any face $f_i$ of $G$ which corresponds to the edges in $E_i$. Define the \textit{quotient graph $\mathcal H_i$ of $H$ with respect to $f_i$} to be a graph obtained from $H$ by contracting all edges in $E(H)\backslash E_i$ and replacing any set of parallel edges by a single edge. So a vertex of $\mathcal H_i$ corresponds to a
connected component of $H\backslash E_i$.

\begin{lemma}\label{lem:quo}
Let $R(G;\mathcal F)$ be the resonance graph of a graph $G$ on a surface with respect to an even-face set $\mathcal F\ne \CMcal F(G)$. Moreover, let $f_i$ be a face of $G$ that corresponds to some edge of a connected component $H$ of $R(G;\mathcal F)$. Then the quotient graph $\mathcal H_i$ with respect to $f_i$ is bipartite.
\end{lemma}

\begin{proof}
Suppose to the contrary that $\mathcal H_i$ has an odd cycle. Then $H$ has a cycle $C$ which contains an odd number of edges corresponding to the face $f_i$. In other words, the face $f_i$ appears an odd number of times in the face sequence corresponding to edges of $C$,  which contradicts Lemma \ref{osnova}. Therefore, $\mathcal H_i$ is bipartite.
\end{proof}

Recall that $\lbrace f_1, \ldots, f_k \rbrace$ is the set of all the faces that correspond to the edges of a connected component $H$ of $\mathcal R(G;\mathcal F)$, and $\mathcal H_i$ is the quotient graph of $H$ with respect to the face $f_i$ for $i \in \lbrace 1, \ldots, k \rbrace$. By Lemma~\ref{lem:quo},
let $(A_i, B_i)$ be the
bipartition of $\mathcal H_i$, and let $\mathcal M_{A_i}$ and $\mathcal M_{B_i}$ be the sets of perfect matchings of $G$ which are vertices of connected components of $H\backslash E_i$ corresponding to vertices of $\mathcal H_i$ in $A_i$ and $B_i$, respectively. 
Define a function $\ell_i: V(H) \rightarrow \lbrace 0,1 \rbrace$ as follows, for any $M \in V(H)$,
\[\ell_i(M)=
\begin{cases}
0; & M \in \mathcal M_{A_{i}} \\
1; & M \in \mathcal M_{B_{i}}.
\end{cases}
\]
Further, define a function $\ell: V(H) \rightarrow \lbrace 0, 1 \rbrace^k$ such that, for any $M \in V(H)$,
\begin{equation}\label{mapping}
\ell(M) = (\ell_1(M), \ldots, \ell_k(M)).
\end{equation}

\begin{theorem}\label{key}
Let $G$ be a graph embedded in a surface, and let $H$ be a connected component of the resonance graph $R(G;\mathcal F)$ of $G$ with respect to an even-face set $\mathcal F\ne \CMcal F(G)$. Then the function $\ell: V(H) \rightarrow \{0, 1\}^k$ defines an embedding of $H$ into a $k$-dimensional hypercube as an induced subgraph.
\end{theorem}

\begin{proof} 
If $G$ has no perfect matching, the result holds trivially. So, in the following, assume that $G$ has a perfect matching.

First, we show that the function $\ell: V(H)\rightarrow \{0,1\}^k$ defined above is injective, i.e., for any $M_1,
M_2\in V(H)$, it holds that $\ell(M_1)\ne \ell(M_2)$ if $M_1\ne M_2$.
Let $P= M_1X_1\ldots X_{t-1}M_2$ be a shortest path of $H$ between $M_1$ and $M_2$. Moreover, let $g_1, \ldots, g_t$ be the faces corresponding to the edges of $P$ such that $g_j$ corresponds to $X_{j-1}X_j$ for $j\in \{1,..., t\}$ (where $X_0=M_1$ and $X_t=M_2$). Note that, some faces $g_i$ and $g_{j}$ may be the same face
for different $i,j\in \{1,...,t\}$.
If every face of $G$ appears an even number of times in the sequence $(g_1, \ldots, g_s)$, then $M_2=M_1\oplus E(g_1)\oplus E(g_2)\oplus \ldots \oplus E(g_s)=M_1$, contradicting that $M_1\ne M_2$. Therefore, there exists a face appearing an odd number of times in the face sequence $(g_1,\ldots, g_s)$.  Without loss of generality, assume the face is $f_i$. By Lemma~\ref{lem:disconn}, two end-vertices of an edge in $E_{i}$ belong to different connected components of
$H\backslash E_{i}$. Since the face $f_i$ appears an odd number of times in the face sequence  $(g_1,\ldots, g_s)$, it follows that if we contract all edges of $P$ not in $E_i$, the resulting walk $P'$ of $\mathcal H_i$ joining the two vertices corresponding to the two components containing $M_1$ and $M_2$ has an odd number of edges. Note that  $\mathcal H_i$ is bipartite by Lemma~\ref{lem:quo}. So one of $M_1$ and $M_2$ belongs to $\mathcal M_{A_{i}}$ and the other belongs to $\mathcal M_{B_{i}}$. So $\ell_i(M_1)\ne \ell_i(M_2)$. Therefore, $\ell(M_1)\ne \ell(M_2)$.

Next, we show that $\ell$ defines an embedding of $H$ into a $k$-dimensional hypercube. It suffices to
show that for any edge $M_1M_2\in E(H)$, it holds that $\ell(M_1)$ and $\ell(M_2)$ differ in exactly one position.
Assume that $M_1M_2$ corresponds to a face $f_i\in \mathcal F$. In other words, the symmetric difference of
two perfect matchings $M_1$ and $M_2$ is the boundary of the face $f_i$.
For any $j \in \{ 1, \ldots, k \}$ and $j \neq i$, the edge $M_1M_2\in E(H\backslash E_j)$ because
$M_1M_2\in E_i$ and $E_i\cap E_j=\emptyset$. Therefore $M_1$ and $M_2$ belong to the same connected component of $H \backslash E_{j}$. Hence $\ell_j(M_1) = \ell_j(M_2)$ for any $j\in \{1,..., k\}$, $j\ne i$. Since $\ell(M_1)\ne \ell(M_2)$, it follows that $\ell(M_1)$ and $\ell(M_2)$ differ in exactly one position, the $i$-th position. Hence, $\ell$ defines
an embedding of $H$ into a $k$-dimensional hypercube.

Finally, we are going to show that $\ell$ embeds $H$ in a $k$-dimensional hypercube as an induced subgraph. It
suffices to show that, for any $M_1, M_2\in V(H)$, $M_1M_2\in E(H)$ if $\ell(M_1)$ and $\ell(M_2)$ differ in exactly one
position. Without loss of generality, assume that $\ell_i(M_1)=0$ and $\ell_i(M_2)=1$ but $\ell_j(M_1)=\ell_j(M_2)$ for any $j\in \{1,...,k\}\backslash \{i\}$. By the definition of the function $\ell$, we have $M_1\in \mathcal M_{A_i}$ and $M_2\in \mathcal M_{B_i}$.
Let $P$ be a path of $H$ joining $M_1$ and $M_2$. 
Then contract all edges of $P$ not in $E_i$  and the resulting walk $P'$ of $\mathcal H_i$ joins two vertices from different partitions of $\mathcal H_i$. So $P'$ has an odd number of edges. In other words, $|P\cap E_i|$ is odd. But, for any $j\in \{1,...,k\}$ and $j\ne i$, contract all edges of $P$ not in $E_j$ and the resulting walk $P''$  joins two vertices from the same partition of $\mathcal H_j$. So $|P\cap E_j|\equiv 0\pmod 2$.
Therefore, for any $e\in E(G)$, the edge $e$ is rotated an odd number of times along path $P$ if $e \in E(f_i)$, but an even number of times if $e \notin E(f_i)$. It follows that $M_1 \oplus M_2 = E(f_i)$, which implies $M_1M_2\in E(G)$. This completes the proof.
\end{proof}

Our main result, Theorem~\ref{main1}, follows directly from Theorem~\ref{key}. \medskip

\section{Clar covers and the cube polynomial }
In this section, we show that the Zhang-Zhang polynomial (or Clar covering polynomial) of a graph $G$ on a surface with respect to an even-face set
$\mathcal F\ne \CMcal F(G)$ is equal to the cube polynomial of
the resonance graph $R(G;\mathcal F)$, which generalizes the main results from papers \cite{zhang-13,be-tr-zi,tr-zp-2} on benzenoid systems, nanotubes (also called tubulenes), and fullerenes. The proof of the equivalence of two polynomials, our main result Theorem \ref{main2}, combines ideas from \cite{zhang-13,be-tr-zi,tr-zp-2} and \cite{khaled}. However, this  general setting of our result requires some new ideas and additional insights into the role and structure
of the resonance graph.  
The following essential lemma generalizes  a result of \cite{zhang-13} originally proved  for benzenoid systems.

\begin{lemma}\label{lem:4cycle}
Let $G$ be a graph embedded in a surface. If the resonance graph $R(G;\mathcal F)$ of $G$ with respect to an even-face set $\mathcal F\ne \CMcal F(G)$ contains a 4-cycle $M_0M_1M_2M_3M_0$,
then $M_0\oplus M_1 = M_2 \oplus M_3$ and $M_0\oplus M_3 = M_1 \oplus M_2$. Further,
the two faces bounded by $M_0\oplus M_1$ and $M_0\oplus M_3$ are disjoint.
\end{lemma}

\begin{proof}
Since $M_0M_1M_2M_3M_0$ is a $4$-cycle in the resonance graph $R(G;\mathcal F)$, let $f_i$ be the face of $G$ such that $E(f_i)= M_i \oplus M_{i+1}$ where $i\in \{0,1,2,3\}$ and subscripts take modulo 4.
Note that 
\begin{equation}\label{eq:1}
E(f_0) \oplus E(f_1) \oplus E(f_2) \oplus E(f_3) =(M_0 \oplus M_1) \oplus (M_1 \oplus M_2) \oplus (M_2 \oplus M_3) \oplus (M_3 \oplus M_0)  = \emptyset.
\end{equation}
Since $M_{i}\ne M_{i+2}$, it follows that  $f_i\ne f_{i+1}$, where $i\in \{0,1,2,3\}$ and subscripts take modulo 4. So $f_i\ne f_{i+1}$
and $f_i\ne f_{i-1}$. By (\ref{eq:1}), every edge on $f_i$ appears on another face $f_j$ with $j\ne i$. It follows that 
$E(f_i)\subseteq \cup_{j\ne i} E(f_j)$ for  $i,j\in \{0,1,2,3\}$. If $f_i$ is distinct from $f_j$ for any $j\ne i$, then all these faces together form a closed surface, which means that $\CMcal F(G)=\mathcal F$, contradicting $\mathcal F\ne \CMcal F(G)$. Therefore, $f_i=f_{i+2}$ for $i\in \{0,1,2,3\}$.
So it follows that $f_0=f_2$ and $f_1=f_3$. In other words, $M_0\oplus M_1 = M_2 \oplus M_3$ and $M_0\oplus M_3 = M_1 \oplus M_2$.

To finish the proof, we need to show that the faces $f_0$ and $f_1$ are disjoint. Suppose to the contrary that $\partial f_0 \cap \partial f_1 \ne \emptyset$. Note that $f_0\ne f_1$. So every component of $\partial f_0\cap \partial f_1$ is a path on at least two vertices. Let $v$ be an end vertex of some component of $\partial f_0\cap \partial f_1$. Therefore, $v$ is incident with three edges $e_1, e_2$ and $e_3$ such that $e_1, e_2\in E(f_0)$ but $e_1, e_3\in E(f_1)$. Since both $f_0$ and $f_1$ are $M_1$-alternating, it follows that $e_1\in M_1$. 
Note that $M_0=M_1\oplus E(f_0)$. So $e_1\notin M_0$. Since $f_1=f_3$, both $f_0$ and $f_3=f_1$ are $M_0$-alternating. Hence $e_1\in M_0$, contradicting $e_1\notin M_0$. This completes the proof.  
\end{proof}


\noindent {\bf Remark.} Lemma~\ref{lem:4cycle} does not hold if  $\mathcal F=\CMcal F(G)$. For example, 
the resonance graph $R(G;\CMcal F(G))$ of the plane graph in Figure~\ref{4_cyc} (left) has a 4-cycle $M_1M_2M_3M_4M_1$ which does not satisfy
the property of Lemma~\ref{lem:4cycle}, where $\CMcal F(G)=\{f_1,...,f_4\}$. \medskip

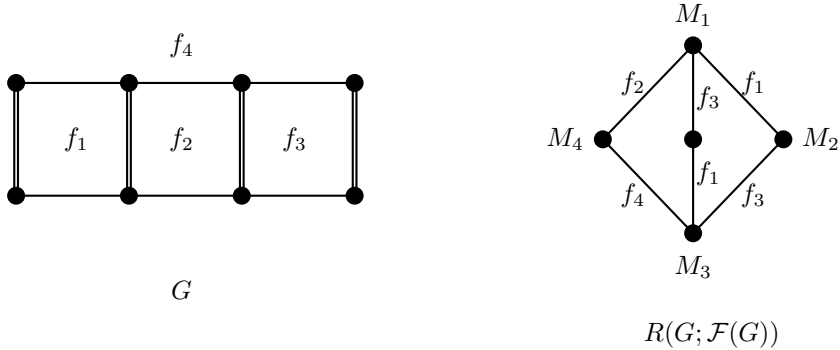
\begin{figure}[!htb]
\begin{center}
\begin{tikzpicture}[thick, scale=1]
\draw [black, thick] (-5,0)--(-3.5,0)--(-2,0)--(-0.5,0);

\draw [black, thick] (-0.5,-1.5)--(-2,-1.5)--(-3.5,-1.5)--(-5,-1.5);

\draw [black, double] (-3.5,0)--(-3.5,-1.5);
\draw [black, double] (-2,0)--(-2,-1.5);
\draw [black, double] (-0.5,0)--(-0.5,-1.5);
\draw [black, double] (-5,0)--(-5,-1.5);

\draw [black, thick] (4,0.5)--(5.2,-.75)--(4,-2)--(2.8,-.75)--(4,0.5);
\draw [black, thick] (4,0.5)--(4,-.75)--(4,-2);

\filldraw [black] (-5,0) circle (3pt);
\filldraw [black] (-5,-1.5) circle (3pt);
\filldraw [black] (-3.5,0) circle (3pt);
\filldraw [black] (-3.5,-1.5) circle (3pt);
\filldraw [black] (-2,0) circle (3pt);
\filldraw [black] (-2,-1.5) circle (3pt);
\filldraw [black] (-0.5,0) circle (3pt);
\filldraw [black] (-0.5,-1.5) circle (3pt);

\filldraw [black] (4,0.5) circle (3pt);
\filldraw [black] (2.8,-.75) circle (3pt);
\filldraw [black] (4,-0.75) circle (3pt);
\filldraw [black] (5.2,-.75) circle (3pt);
\filldraw [black] (4,-2) circle (3pt);

\node [] at (-2.8,0.5) { $f_4$};
\node [] at (-2.8,-0.75) { $f_2$};
\node [] at (-4.2,-0.75) { $f_1$};
\node [] at (-1.3,-0.75) { $f_3$};
\node [] at (-2.8,-2.75) {$G$};

\node [] at (4,0.9) { $M_1$};
\node [] at (2.3,-.75) { $M_4$};
\node [] at (5.7,-.75) { $M_2$};
\node [] at (4,-2.45) { $M_3$};

\node [] at (4.8,0) { $f_1$};
\node [] at (4.8,-1.5) { $f_3$};
\node [] at (3.2,-1.5) { $f_4$};
\node [] at (3.2,0) { $f_2$};
\node [] at (4.2,-0.2) { $f_3$};
\node [] at (4.2,-1.2) { $f_1$};
\node [] at (4.25, -3.35) { $R(G;\CMcal  F(G))$};

\end{tikzpicture}
\caption{{\small The resonance graph  where the double edges of $G$ form $M_1$.}}\label{4_cyc}
\end{center}
\end{figure}

Now, we are going to prove Theorem~\ref{main2}. \medskip

\noindent{\bf Proof of Theorem~\ref{main2}.} Let $G$ be a graph on a surface and let $ R(G;\mathcal F)$ be the resonance graph of 
$G$ with respect to  $\mathcal F$. 
If $G$ has no perfect matching, then the result holds trivially. So, in the following, we always assume that 
$G$ has a perfect matching.

For an
nonnegative integer $k$, let   $\mathcal Z_k(G,\mathcal F)$ be the set of all Clar covers of $G$ with exactly $k$ faces such that all these faces are included in $\mathcal F$, and let $\mathcal Q_k(R(G;\mathcal F))$ be the set of all labeled subgraphs of $R(G;\mathcal F)$ that are isomorphic to the $k$-dimensional hypercube.
For a Clar cover  $S \in \mathcal Z_k(G;\mathcal F) $, let  $M_1,M_2,\ldots, M_t$ be all the perfect matchings
of $G$ such that all faces in $S$ are $M_i$-alternating and
all isolated edges of $S$ belong to $M_i$ for all $i\in \{1,..., t\}$.
Define a mapping
$$m_k:\, \mathcal Z_k(G,\mathcal F) \longrightarrow \mathcal Q_k( R(G;\mathcal F))$$
such that $m_k(S)$ is the subgraph of $R(G;\mathcal F)$ induced by the vertex set $\{M_1,M_2,\ldots, M_t\}$.
Since the subgraph induced by $\{M_1,\ldots, M_t\}$ is unique, the mapping $m_k$ is well-defined, which follows from  the following claim. \medskip

\noindent{\bf Claim~1.} {\sl For each  Clar cover $S \in \mathcal Z_k(G;\mathcal F) $, the image $m_k(S) \in \mathcal Q_k( R(G;\mathcal F))$.} \medskip

\noindent{\it Proof of Claim~1.}
It is sufficient to show that $m_k(S)$ is isomorphic to the $k$-dimensional hypercube $Q_k$. Let
$f_1,f_2,\ldots,f_k$ be the faces in the Clar cover $S$. Then $\{f_1,\ldots, f_k\}\subseteq \mathcal F$. 
So each $f_i$ with $i\in \{1,\ldots, k\}$ is even and hence has two perfect matchings labelled by ``0" and ``1" respectively. For any vertex $M$ of $m_k(S)$, let
$b(M) =(b_1,b_2,\ldots, b_k)$, where $b_i=\alpha $ if $M\cap E(f_i)$ is the perfect matching of $\partial f_i$ with label $\alpha \in \{0,1\}$  for each $i\in \{1,2,\ldots,k\}$.
It is obvious that $b:V(m_k(S))\rightarrow V(Q_k)$ is a bijection. For $M' \in V(m_k(S))$, let
$b(M')=(b_1',b_2',\ldots, b_k')$. If $M$ and $M'$ are adjacent in $m_k(S)$ then $M\oplus M'=E(f_i)$ for some $i\in\{ 1,2,\ldots, k\}$. Therefore, $b_j=b_j'$ for each $j \neq i$ and $b_i \neq b_i'$, which implies
 $(b_1,b_2,\ldots,  b_k)$ and $(b_1',b_2',\ldots, b_k')$ are adjacent in $Q_k$. Conversely, if
  $(b_1,b_2,\ldots, b_k)$ and $(b_1',b_2',\ldots, b_k')$ are adjacent in $Q_k$, it follows that $M$ and
  $M'$ are adjacent in $m_k(S)$. Hence $b$ is an isomorphism between $m_k(S)$ and $Q_k$. This completes the
  proof of Claim~1. \medskip

In order to show $\ZZ_{\mathcal F}(G,x)=\QQ( R(G;\mathcal F),x)$, it suffices to show that mapping $m_k$ is bijective for any $k$. Note that in the case of $k=0$, a Clar cover $S$ is a perfect matching of $G$ and hence $m_k(S)$ is a vertex of $  R(G;\mathcal F)$. So the mapping $m_k$ is obviously bijective for $k=0$. In the following, assume that  $k$ is a positive integer. 

First, we show that $m_k$ is injective.  
Let $S$ and $S'$  be two different Clar covers from $\mathcal Z_k(G;\mathcal F)$. If $S\cap \mathcal F =S'\cap \mathcal F$, then the isolated edges of $S$ and $S'$ are different. So a perfect matching of $S$ is different from a perfect matching of $S'$.  Therefore, the vertex sets of $m_k(S)$ and $m_k(S')$ are disjoint. Hence $m_k(S)$ and $m_k(S')$ are disjoint induced subgraphs of $ R(G;\mathcal F)$. So $m_k(S)\neq m_k(S')$. Now suppose that $S\cap \mathcal F \ne S'\cap \mathcal F$. Note that $|S\cap \mathcal F|=|S'\cap \mathcal F|=k$.  So $S\cap \mathcal F$ has a face $f \notin S'\cap \mathcal F$. Note that the faces adjacent to $f$ do not all belong to $S'$ since the faces in $S'$ are independent. Hence the face $f$ contains at least one edge $e$ which does not belong to  $S'$. From the definition of the function $m_k$, the edge $e$ does not belong to those perfect matchings of $G$ that correspond to the vertices of $m_k(S')$. For any perfect matching $M$ corresponding to a vertex of $m_k(S)$, the face $f$ is $M$-alternating. Hence either $M$ or $M'=M\oplus E(f)$ contains $e$. Without loss of generality, assume that $e\in M'$. So $M'$ is not a vertex of $M_k(S')$. Since both $M$ and $M'$ are perfect matchings of $S$, both $M$ and $M'$ are vertices of $m_k(S)$. So $m_k(S) \neq m_k(S')$. This shows that $m_k$ is injective.

In the following, we are going to show that $m_k$ is surjective.
Let  $Q \in \mathcal Q_k(R(G;\mathcal F))$, isomorphic to a $k$-dimensional hypercube. Then every vertex $u$ of $Q$ can be represented by a binary string $(u_1, u_2, \ldots, u_k)$ such that two vertices of $Q$ are adjacent in $Q$ if and only if their binary strings differ in precisely one position. Label the vertices of $Q$ by $M^0 = (0,0,0, \ldots, 0)$, $M^1 = (1,0,0, \ldots, 0)$, $M^2 = (0,1,0, \ldots, 0)$, \ldots, $M^k = (0,0,0, \ldots, 1)$. So $M^0M^i$ is an edge of $R(G;\mathcal F)$ for every $i\in \{1,\ldots, k\}$. By definition of $ R(G;\mathcal F)$, the symmetric difference of  perfect matchings  $M^0$ and $M^i$ is the boundary of an even  face in $\mathcal F$, denoted by $f_i$. Then  we have a set of faces $\lbrace f_1, \ldots, f_k \rbrace\subseteq \mathcal F$. Note that $f_i\ne f_j$ for $i,j \in \lbrace 1, \ldots, k \rbrace$ and $i \neq j$ since $M^i \ne M^j$. Hence, all faces in $\{f_1,\ldots f_k\}$ are distinct. In order to show that $m_k$ is surjective, it is sufficient to show that $G$
has a Clar cover $S$ such that $S\cap \mathcal F=\{f_1,\ldots, f_k\}$.  \medskip

\noindent {\bf Claim 2.} {\sl All faces in $\{f_1,\ldots, f_k\}$ are pairwise disjoint.} \medskip

\noindent{\em Proof of Claim 2:}
Let $f_i,f_j \in \lbrace f_1, \ldots, f_k \rbrace$ with $i \neq j$ and let $W$ be a vertex of $Q$ having exactly two $1$'s which are in the $i$-th and $j$-th position. Then $M^0M^iWM^jM^0$ is a 4-cycle such that $E(f_i)=M^0\oplus M^i$ and $E(f_j)=M^0\oplus M^j$. Then by Lemma \ref{lem:4cycle}, it follows that $f_i$ and $f_j$ are disjoint. 

\medskip

By Claim 2, we only need to show that $G-\cup_{i=1}^k V(f_i)$ has a perfect matching $M$ so that $S=M\cup \{f_1,\ldots, f_k\}$ is a Clar cover of $G$. Consider the perfect matching $M^0$ corresponding to the vertex of $Q$ labelled by the string with $k$ zero's. Recall that $E(f_i)=M^0\oplus M^i$
and hence every $f_i$ is $M^0$-alternating for any $i\in \{1,2,\ldots, k\}$. Therefore, $M:=M^0\backslash (\cup_{i=1}^k E(f_i))$ is a perfect 
matching of $G-\cup_{i=1}^k V(f_i)$. So $S=M\cup \{f_1,\ldots, f_k\}$ is a Clar cover of $G$ such that $m_k(S)=Q$. This completes the proof of that $m_k(G)$ is surjective.

From the above, $m_k$ is a bijection between the set of all Clar covers of $G$ with $k$ facial cycles and the set of all labeled subgraphs isomorphic to the $k$-dimensional hypercube for any integer $k$. Therefore, we have $\ZZ_{\mathcal F}(G,x)=\QQ(R(G;\mathcal F),x)$ and this completes the proof of Theorem~\ref{main2}. \qed

\section{Concluding remarks}

Let $G$ be a graph embedded in a surface $\Sigma$ and let $\mathcal F$ be an even-face set. Assume that $H$ is a connected component of $R(G;\mathcal F)$. Let $G_H$ be the subgraph of $G$ induced by the faces corresponding to edges of $H$. 

\begin{proposition}
Let $G$ be a graph embedded in a surface $\Sigma$ and let $\mathcal F$ be an even-face set. If a perfect matching $M$ is  a vertex of a connected component $H$ of $R(G;\mathcal F)$, then $M\cap E(G_H)$ is a perfect matching of $G_H$.
\end{proposition}
\begin{proof}
Let $M$ be a perfect matching corresponding to a vertex of $H$. Suppose to the contrary that $M\cap E(G_H)$ is not a perfect matching of $G_H$. Then $G_H$ has a vertex $v$ which is not covered by $M\cap E(G_H)$. Let $f\in \mathcal F$ be a face containing $v$, which corresponds to an edge of $H$. Then $f$ is $M'$-alternating for some perfect matching $M'$ which is a vertex of $H$. Since  $H$ is connected, there is a path $P$ of $H$ joining $M$ and $M'$. Assume that the faces corresponding to the edges of $P$ are $f_1,\ldots, f_k$. Then  $M=M'\oplus E(f_1)\oplus \ldots \oplus E(f_k)$. Note that $E(f_i)\subseteq E(G_H)$ for all $i\in \{1,\ldots, k\}$. So $v$ is incident with an edge in $M\cap E(G_H)$, a contradiction. This completes the proof.
\end{proof}

For a connected component $H$ of $R(G;\mathcal F)$, if the union of all faces corresponding to edges of $H$ is homeomorphic to a closed disc, then $G_H$ with the embedding inherited from the embedding from $G$ in $\Sigma$ is a plane elementary bipartite graph. By Theorem \ref{thm:plane}, we have the following proposition. 

\begin{proposition}
Let $G$ be a graph on a surface $\Sigma$ and  $\mathcal F$ be an even-face set. Assume that $H$ is a connected component of $R(G;\mathcal F)$. If the union of all faces corresponding to edges of $H$ is homeomorphic to a closed disc, then $H$ is the covering graph of a distributive lattice. 
\end{proposition}

It has been evident in \cite{tr-zp-2} that, if the subgraph induced by faces in $\mathcal F$ is non-bipartite, a connected component of $R(G;\mathcal F)$ may not be the covering graph of a distributive lattice.  But the condition in Proposition 5.2 is not a necessary condition. It has been shown in \cite{tr-zi} that a connected component of an annulus graph (a plane graph excluding two faces) could be the covering graph of a distributive lattice. 
It is natural to ask what is the necessary and sufficient condition for $\mathcal F$ so that every connected component of $R(G;\mathcal F)$ is the covering graph of a distributive lattice.  

But so far, in all examples we have, a connected component of $R(G;\mathcal F)$ is always a median graph. Therefore, we risk the following conjecture.

\begin{conj}
Let $G$ be a graph embedded in a surface and let $\mathcal F \ne \CMcal F(G)$ be an even-face set. Then every connected component of the resonance graph $ R(G;\mathcal F)$ is a median graph.
\end{conj}

In order to prove the above conjecture or improve Theorem~\ref{main1}, some new idea different from what we have in the proof of Theorem~\ref{key} is required. Also, it would be interesting to show that one can embed a connected component of $R(G;\mathcal F)$ into a hypercube as an isometric subgraph, which would be a weaker result.


\begin{figure}[h!]
\begin{center}
\includegraphics[scale=0.59]{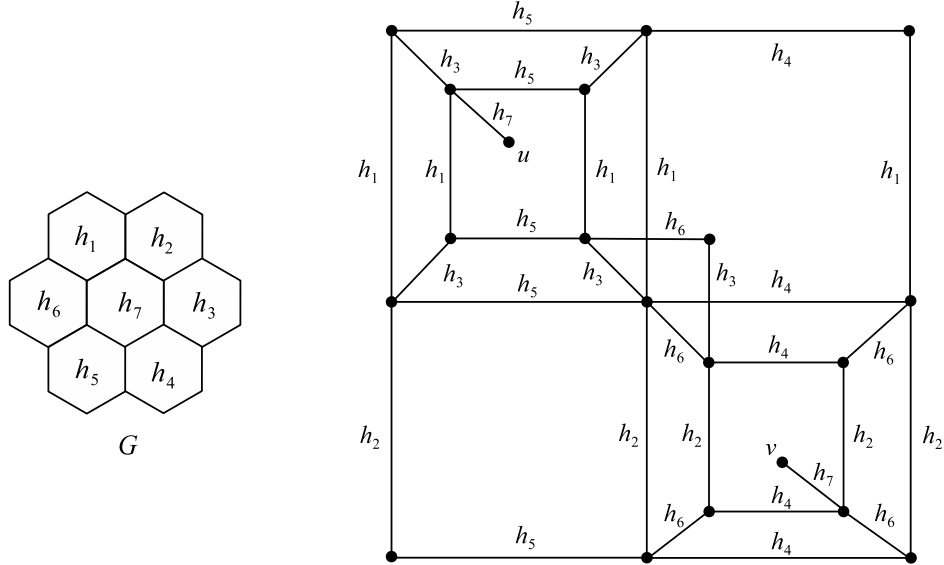}
\end{center}
\caption{\label{koronen_res}  A plane graph $G$ and the resonance graph $R(G;\mathcal F)$.}
\end{figure}

Note that, the embedding function $\ell$ given in Equation (\ref{mapping}) and used in the proof of Theorem~\ref{key} is not always an isometric embedding. For example, let $G$ be a plane graph as shown on the left in Figure \ref{koronen_res} and let $\mathcal F$ be the set of all inner faces, i.e. $\mathcal F=\{h_1,\ldots,h_7\}$. Then the resonance graph $R(G;\mathcal F)$ is the graph shown on the right in Figure \ref{koronen_res}. Let $E_i$ be the set of edges of $R(G;\mathcal F)$ corresponding to the face $h_i$ for $i\in \{1,\ldots, 7\}$. Note that the resonance graph $R(G;\mathcal F)\backslash E_i$ for $i \in \lbrace 1, \ldots, 6 \rbrace$ has exactly two connected components and the vertices $u$ and $v$ of $R(G;\mathcal F)$ belong to different connected components. Therefore, the binary strings  $\ell(u)$ and $\ell(v)$ differ in the first six positions. But the subgraph $R(G;\mathcal F) \backslash E_7$ has three connected components and  vertices $u$ and $v$ belong to two components that are not connected by any edge in the quotient graph. Therefore, the binary strings  $\ell(u)$ and $\ell(v)$ have the same number in the last position. So $\ell(u)$ and $\ell(v)$ differ exactly in six positions. However, the distance between $u$ and $v$ in $R(G;\mathcal F)$ is eight. So the embedding $\ell$  is not isometric. Therefore, the proof of Theorem~\ref{key} 
may not be adapted to show that a connected component of $R(G;\mathcal F)$ is a partial cube, nor a median graph.

\end{document}